\documentclass[a4paper,11pt]{amsart}
\usepackage{amsmath,amsthm,amssymb,amsfonts}
\usepackage[colorlinks=true, allcolors=blue]{hyperref}
\usepackage{bm}
\usepackage[usenames, dvipsnames]{color}
\usepackage{inputenc}

\usepackage[english]{babel}
\usepackage{tikz}
\usepackage{url}

\providecommand{\U}[1]{\protect\rule{.1in}{.1in}}
\newtheorem{theorem}{Theorem}

\newtheorem{lemma}[theorem]{Lemma}

\newtheorem{proposition}[theorem]{Proposition}

\usepackage[numbers,sort]{natbib}

\newcommand{\C}{\mathbb{C}}

\newcommand{\E}{\mathbb{E}}

\newcommand{\N}{\mathbb{N}}

	\renewcommand{\P}{\mathbb{P}}

\newcommand{\R}{\mathbb{R}}
\renewcommand{\S}{\mathbb{S}}
\newcommand{\T}{\mathbb{T}}

\newcommand{\Z}{\mathbb{Z}}

\newcommand{\ep}{\epsilon}
\newcommand{\cN}{\mathcal{N}}

\newcommand{\Kq}[0]{{K}_{q}}
\newcommand{\n}[0]{\mathsf N}

{}

\numberwithin{equation}{section}
\numberwithin{theorem}{section}

\theoremstyle{definition}

\begin{document}
%%%%%%%%%%%%%%%%%%%%%%%%%%%%%%%%%%%%%%%%%%%%%%%%
%\linenumbers
%%%%%%%%%%%%%%%%%%%%%%%%%%%%%%%%%%%%%%%%%%%%%%%%
%\title{Stable twisted states are generic in ring-like graphs for the Kuramoto model.}
\title[Kuramoto model in random geometric graphs]{The energy landscape of the Kuramoto model in random geometric graphs in a circle}
%%%%%%%%%%%%%%%%%%%%%%%%%%%%%%%%%%%%%%%%%%%%%%%%

%%%%%%%%%%%%%%%%%%%%%%%%%%%%%%%%%%%%%%%%%%%%%%%%
\author[C. De Vita]{Cecilia De Vita}
\address{Departamento de Matem\'atica\hfill\break \indent Facultad de Ciencias Exactas y Naturales\hfill\break \indent Universidad de Buenos Aires\hfill\break \indent IMAS-UBA-CONICET\hfill\break \indent Buenos Aires, Argentina}
\email{cdevita@dm.uba.ar}

\author[J.F. Bonder]{Juli\'an Fern\'andez Bonder}
\address{Departamento de Matem\'atica\hfill\break \indent Facultad de Ciencias Exactas y Naturales\hfill\break \indent Universidad de Buenos Aires\hfill\break \indent Instituto de C\'alculo-UBA-CONICET\hfill\break \indent Buenos Aires, Argentina}
\email{jfbonder@dm.uba.ar}

\author[P. Groisman]{Pablo Groisman}
\address{Departamento de Matem\'atica\hfill\break \indent Facultad de Ciencias Exactas y Naturales\hfill\break \indent Universidad de Buenos Aires\hfill\break \indent IMAS-UBA-CONICET\hfill\break \indent Buenos Aires, Argentina}
%\address{NYU-ECNU Institute of Mathematical Sciences at NYU Shanghai}
\email{pgroisma@dm.uba.ar}

%%%%%%%%%%%%%%%%%%%%%%%%%%%%%%%%%%%%%%%%%%%%%%%%


\begin{abstract}
We study the energy function of the Kuramoto model in random geometric graphs defined in the unit circle as the number of nodes diverges. We prove the existence of at least one local minimum for each winding number  $q\in \Z$ with high probability. Hence providing a large family of graphs that support patterns that are generic. These states are in correspondence with the explicit twisted states found in WSG and other highly symmetric networks, but in our situation there is no explicit formula due to the lack of symmetry.
The method of proof is simple and robust. It allows other types of graphs like $k-$nn graphs or the boolean model and holds also for graphs defined in any simple closed curve or even a small neighborhood of the curve and for weighted graphs. It seems plausible that the method can be extended also to higher dimensions, but a more careful analysis is required.
\end{abstract}

\maketitle

\section{Introduction}

The study of local minima and the whole geometry of high-dimensional random non-convex functions is highly relevant in areas as diverse as deep-learning, statistical mechanics, complex networks and synchronicity.

Phase synchronization of systems of coupled oscillators is a phenomenon that has attracted the mathematical and scientific community because of its intrinsic mathematical interest \cite{chiba2016mean, medvedev2018continuum, bertini2014synchronization, coppini2020law} and its ubiquity in technological, physical and biological models \cite{mirollo1990synchronization, winfree1967biological, acebron2005kuramoto, bullo2020lectures, arenas2008synchronization, dorflerSurvey, strogatz2004sync, strogatz2000kuramoto}.

The Kuramoto model is one of the most popular models for describing synchronization of a system of coupled oscillators. The model has been studied both by means of rigorous mathematical proofs and heuristics arguments and simulations in different families of graphs. Here we focus on the first type of evidence.

We consider graphs $G_n = G=(V,E)$ where the set of nodes $V=\{x_0, x_1, \dots, x_{n-1}\}\subset \mathbb S^1:= \{z \in \mathbb C \colon |z|=1\}$ is a sample of $n$ i.i.d uniform random variables. The distance between two nodes is given by the geodesic distance in $\mathbb S^1$, that we denote with $d(x_i,x_j):= \cos^{-1}(x_i \cdot x_j)$. For convenience, we assume that the nodes $V$ are labeled counterclockwise with $\arg(x_0)=0$ and we denote $x_n:=x_0$.

The {\em random geometric graph} in the circle $\S^1$ with parameters $n,\ep_n$ is the graph that has $V$ as the set of nodes in which we declare $\{x_i,x_j\} \in E$ (we denote this by $i \sim j$) if and only if $d(x_i,x_j)< \ep_n$.

We can think of $x_0, \dots, x_{n-1}$ as points in $[0,2\pi]$ and $d(x_i,x_j)$ as the one-dimensional distance $|x_i-x_j|$ with the convention that everything is understood mod $2\pi$.

For the sequence of random geometric graphs defined above, we are going to work in the regime
\begin{equation}
\label{condition}
n\ep_n^{2}\to 0, \qquad  \frac{n\ep_n}{\log n} \to \infty,  \qquad \text{ as }n\to\infty.
\end{equation}
The first condition implies $\ep_n \to 0$, which is important to obtain Proposition \ref{conv.grad} below (this proposition does not hold if $\ep_n \nrightarrow 0$). It is also used to conclude the main theorem. However, we expect the conclusion of our main theorem to hold even without Proposition \ref{conv.grad} (but some bound from above is needed on $\ep_n$ to avoid high connectivity that leads to global synchronization \cite{wiley2006size}). The second condition is required to guarantee connectivity of the graph with high probability and so, it can't be removed without altering the behavior of the system. Observe that both conditions are verified for any sequence of the form $\ep_n \approx n^{-a}$ with $1/2 < a <1$.

Let $u^{n}\colon [0,\infty)\times V\to\R$ be the unique solution to a system of $n$ homogeneous
Kuramoto equations %on $\T^d$:
\begin{align}
\label{eq:km}
\begin{cases}
\displaystyle{\frac{d}{dt}}u^{n}(t,x_i) = \frac1n \displaystyle{\frac{1}{\ep_n^2 \mathbb{E}(N_{i})}}\sum_{j\sim i}\sin \left(u^{n}(t,x_j)-u^{n}(t,x_i)\right),\\
u^{n}(0,x_i)=\bar u^n(x_i), \quad\quad  i=0,1,...,n-1.
\end{cases}
\end{align}

The random integer $N_i$ denotes the number of neighbors of $x_i$.
%We also call the set of such neighboring points $\cN(i)\subset\{0,1,...,n-1\}\backslash\{i\}$.
To lighten notation we call $u^n_i =u^n(t,x_i)$ and we also omit the dependence on $n$ if it is not necessary. Observe that $\mathbb{E}(N_{i}) = n\ep_{n}/\pi$. Hence, the equation in \eqref{eq:km} can be rewritten as
\begin{align*}
	\displaystyle{\frac{d}{dt}}u_{i} =  \displaystyle{\frac{\pi}{n^{2} \ep_n^3}} \sum_{j\sim i}\sin \left(u_{j} - u_{i}\right).
\end{align*}

To understand this scaling it is worth to note that \eqref{eq:km} defines a gradient system. It is a direct computation to see that
\[
 \dot u = -\nabla E_n(u),
\]
for
\[
E_n(u)= E_n(u_0,\dots, u_{n-1}) =  \frac{\pi}{2 n^{2} \ep_{n}^{3}}\sum_{i=1}^{n} \sum_{j\sim i}(1 - \cos (u_j-u_i)).\\
\]
We will see that with this scaling $E_{n}(u)$ has a nontrivial limit as $n \to \infty$. Since for large $n$ the sum in $E_{n}(u)$ involves approximately $n^{2}\ep_{n}/\pi$ terms of order $\ep_{n}^{2}$ (when $u$ is smooth), it makes sense to think that this is the correct scaling. The sine function in \eqref{eq:km} can be replaced by an odd $2\pi-$periodic symmetric smooth function $J$ with Taylor expansion $J(\theta)=\theta + o(\theta^2)$ without altering the conclusion of our main result.

Note that \eqref{eq:km} is invariant under shifts: $u= (u_i^n(t))_{0\le i \le n, t \ge 0}$ is a solution if and only the same holds for $u_c = u+c$ for any $c \in \S^1$. In particular, for every $c$, the set $\mathcal I_c = \{u\colon \sum_{i=1}^n u_i^n = c\}$ is invariant and they correpond to ``copies'' of the same dynamical system. Due to this fact, throughout the rest of the paper we assume $c=0$, which correspond to understand \eqref{eq:km} in the orthogonal space of $(1,1, \dots, 1)$, which is invariant.

Our interest in the Kuramoto model in graphs with this structure is threefold: on the one hand this kind of graphs is relevant to model several natural situations in which spatial considerations are important to determine the strength of the links between oscillators. On the other hand they form a large family of model networks with persistent behavior (robust to small perturbations) for which we expect to observe patterns.

Last but not least, there has been a recent interest to understand the behavior of the Kuramoto model on diverse models of random and non-random graphs \cite{ling2019landscape, abdalla2022guarantees, abdalla2023expander, kassabov2022global}. The main goal is to decide if the networks foster synchronization or not. In \cite{abdalla2023expander} the authors have recently shown that in expander graphs and in particular in Erd\H{o}s-R\'eny graphs above the connectivity threshold, synchronization occurs with high probability as $n\to\infty$. Our results can be seen as a complement of those in the sense that we are exhibiting a class of random graphs that are not expanders for which global synchronization fails. Up to our knowledge this is the first rigorous proof of non-synchronization in random {geometric} graphs.

Twisted states have been defined for particular classes of graphs as explicit equilibria of \eqref{eq:km}. They have been shown to be stable equilibria in WSG networks (rings in which each node is connected to its $k$ nearest neighbors on each side \cite{wiley2006size}), in Cayley graphs and in random graphs with a particular structure \cite{MedvedevStability}. They have also been studied in small-world networks \cite{MedvedevSmallWorld} and in the continuum limit \cite{MedvedevContTwist} among others.

Our notion of {\em twisted state} is a bit different since we don't expect to find explicit equilibria in our context besides complete synchronization. We think of them as stable equilibria that can be identified in some way with the functions $u_q(x_i)=qx_i$. Precise definitions are given in the next section. We remark that we are considering functions that take values in $\S^1$ rather than $\R$. Alternatively, we can think of them as functions $u\colon[0,2\pi]\to \R$ with $u(2\pi)=u(0) + 2q\pi$ for some $q\in \Z$.

Situations in which the twisted states are explicit are not expected to be robust and persistent. Our interest is to find twisted states that are generic in some sense and for this same reason we do not expect us to be able to compute them explicitly.

Remark that as far as we know, in most of the literature that give rigorous proofs about existence of twisted states they are computed explicitly by exploiting graph symmetries and the issue is to prove their stability. Here (and in most typical real situations with spatial structure and local interactions) the issue is to prove their existence. We are going to get the stability for free.

The Kuramoto model in random geometric graphs has been studied in \cite{abdalla2022guarantees}. In that work the authors are interested in the optimization landscape of the energy function determined by \eqref{eq:km} as well as we are here, but they work on a different regime: in their setting the graphs are constructed on the sphere $\mathbb S^{\ell-1}$ rather than in the circle and $\ell \to \infty$ as $n\to \infty$. In that context, they obtain guarantees for global spontaneous synchronization (i.e. the global minimum $\theta_1=\theta_2 = \cdots = \theta_n$ is the unique local minima of the energy). This is pretty different from our situation as we will see.

Besides the Kuramoto model, our work enters in the framework of random non-convex optimization, which is relevant not just in the study of dynamics of complex networks but also in deep-learning and statistical mechanics. In the first case due to the fact that most of modern learning algorithms (i.e., artificial neural networks) rely on the adequate optimization of a loss function which is typically highly non-convex and random \cite{LeCunBenArous, LeCunBenArous2, Baskerville, NguyenHein}.

In particular, our results show that for this kind of random energies, while the energy at a typical point diverges to infinity with the size of the graph, at any given local minima is of order one. Moreover, this implies that every local minima is (asymptotically) smooth, which is related to the mysterious phenomenon of {\em implicit regularization} \cite{BelkinPNAS,BelkinPNAS2,BelkinActa}.

We are going to state our results for random geometric graphs but they can be exported straightforwardly to different kinds of graphs defined in any closed and simple curve like $k-$nn graphs or even deterministic graphs. We discuss this in Section \ref{discussion}.

For a continuous function $u\colon[0,2\pi]\to \R$ with $u(2\pi)=u(0) + 2q\pi$ for some $q\in \Z$ we define its index by $I(u)=q$. If $u$ is defined only in a discrete set $\{x_0, \dots, x_{n-1}\}$ we define its index as the index of its linear interpolation (a more precise definition is given below).

Our main result reads as follows.
\begin{theorem} \label{main.thm}
For each $q\in \Z$ we have,
\[
\lim_{n\to\infty}\P\left(\eqref{eq:km} \text{ has an asymptotically stable equilibrium with index $q$} \right)  = 1.
\]
\end{theorem}
We remark that the idea of using winding numbers (the index) to identify the non-synchronous states in this kind of context goes back at least to \cite{wiley2006size}.

To prove this theorem we first consider in Section \ref{geometry.space} a partition of the domain of $E_n$, {the $n$-dimensional torus} {$\mathbb{T}^{n} := \underbrace{\mathbb{S}^{1} \times \ldots \times \mathbb{S}^{1}}_{n}$}. Next, in Section \ref{proof.main.thm} we prove Proposition \ref{conv.grad} which is one of the main ingredients and then Theorem \ref{main.thm}.

\section{Geometry of the space {$\mathbb{T}^{n}$}}
\label{geometry.space}

For a rectifiable closed curve $\gamma : [a,b] \to \mathbb{C}$ that does not contain the origin we define the index (or winding number) of $\gamma$ (around the origin) as the  total number of times that the curve travels counterclockwise around $0$. More precisely,
\begin{align*}
    I(\gamma) = \frac{1}{2\pi \mathrm i} \int_{\gamma} \frac{dz}{z}.
\end{align*}

For each $\mathbf{z} \in \mathcal{D} = \{\mathbf{z} \in \mathbb{T}^{n} : z_{j} \neq -z_{j-1}, \, 1 \leq j \leq n\}  \subset \C^n$, consider $\gamma_{\mathbf z} = \sum_{j=1}^{n} \gamma_{j}$ where $\gamma_{j}$ is the geodesic from $z_{j-1}$ to $z_{j}$ in $\S^1$, $z_{n} = z_{0}$ and $\gamma_i + \gamma_j$ is the curve that results from concatenating $\gamma_i$ and $\gamma_j$ in the given order. Observe that since $z_j \ne -z_{j-1}$, $\gamma_{\mathbf z}$ is a well-defined picewise differentiable closed curve in $\S^1$ (otherwise the geodesic from $z_{j-1}$ to $z_j$ is not unique). We abbreviate notation by writing  $I({\mathbf z}):=I(\gamma_{\mathbf z})$.
%\ceci{for every $\mathbf{z} \in \mathcal{D} = Dom(I)$}.

If we write $z_{j} = e^{{\rm i} \theta_{j}}$ for some $\theta_{j} \in [0, 2\pi)$, then
\begin{align*}
    I({\mathbf z}) &= \frac{1}{2\pi {\rm i}} \int_{\gamma_{\mathbf z}} \frac{dz}{z}\\
    &= \frac{1}{2\pi {\rm i}} \sum_{j=1}^{n} \int_{\gamma_{j}} \frac{dz}{z} \\
%    &= \frac{1}{2\pi {\rm i}} \sum_{j=1}^{n} \int_{\theta_{j-1}}^{\theta_{j}} \frac{{\rm i}e^{{\rm i}t}}{e^{{\rm i}t}}\, dt \\
    &= \frac{1}{2\pi} \sum_{j=1}^{n} \theta_{j} \circleddash \theta_{j-1}.
\end{align*}
Here $\theta_j \circleddash \theta_{j-1}$ is the signed length of the geodesic from $z_{j-1}$ to $z_j$. If for each $\theta,\theta'$ we choose $\bar \theta$ and $\bar \theta' $ such that $e^{{\rm i} \bar \theta} = e^{{\rm i} \theta}$, $e^{{\rm i} \bar \theta'}  = e^{{\rm i} \theta'}$ with
$-\pi < \bar \theta -\bar \theta' < \pi$, it can be computed as
\[
\theta \circleddash \theta'= \bar \theta - \bar \theta'.
\]
Sometimes we will slightly abuse notation by writing $I(\theta_0,\dots,\theta_{n-1})$ instead of $I(z_0,\dots, z_{n-1})$. This is not a problem since the value of $I$ is independent of the choice of $\theta_0, \dots,\theta_{n-1}$.

Observe that the set of points $\mathbf z \in {\mathbb{T}^{n}}$ for which the index $I(\mathbf z)$ is well defined is open and that the function $I$ is continuous in its domain $\mathcal D$ and integer valued. Hence it is constant in each connected component and in fact the sets
 \begin{align*}
     \Kq := \{\mathbf z \in {\mathbb{T}^{n}} : I(\mathbf z) = q\},
 \end{align*}
define the connected components of $\mathcal D$. Note that each $K_q$ is open and $\partial\Kq = \{\mathbf z \in \overline{\Kq} :  I(\mathbf z) \text{ is not defined} \}$. So, we have the decomposition
\[
 {\mathbb{T}^{n}} = \bigcup_{q \in \Z} K_q    \cup  \left(\bigcup_{q \in \Z} \partial K_q \right).
\]
Remark that for a given $n$, the sets $K_q=\emptyset$ for $|q|> \lfloor \frac{n-1}{2} \rfloor$. Also remark that for $|q|\le \lfloor \frac{n-1}{2} \rfloor$ we have $\partial K_q \cap \partial K_{q'} \ne \emptyset$. In fact the point $(0,\pi,0,\pi,\dots) \in \partial K_q$ for every $q \le \lfloor \frac{n-1}{2} \rfloor$.

We will prove that for each $q \in \Z$, the energy $E_n$ restricted to $K_q$ attains a minimum with high probability as $n\to \infty$. Since the sets $K_q$ are open, this minima are forced to be local minima of $E_n$.

% \begin{proof}
%     Let $0 < \epsilon < 1$. Since $I$ is an integer-valued continuous function, we have that
%     \begin{align*}
%         \Kq &= I^{-1}(\{q\}) \\
%         &= I^{-1}\{(q - \epsilon, q + \epsilon)\}
%     \end{align*}
%     is an open set.
%
%     PROBAR LO DEL BORDE
%
% \end{proof}

%\end{proposition}

\section{Proof of the main theorem}
\label{proof.main.thm}

To prove Theorem \ref{main.thm} we will need the following proposition applied to the functions $u_{q}(x)= qx$, but we state it for general smooth functions $u$ due to its independent interest.

\begin{proposition}
\label{conv.grad}
Assume $u\in C^2([0,2\pi],\R)$, then
\[
\lim_{n\to \infty} E_n(u) = \frac{1}{12\pi} \int_0^{2\pi}|u'(x)|^2\, dx, \qquad \text{in probability.}
\]
\end{proposition}

We expect this convergence to hold almost surely but a proof of that fact would be more involved and it is not required for our arguments.

\begin{proof} We will need to make use of a Poissonization/de-Poissonization argument. The goal of this type of argument is to work in a setting with more independence. In our model, if we consider two disjoint intervals in the circle $\S^1$, the number of points in each of them are random variables that are not independent. If instead of considering $n$ i.i.d. points uniformly distributed we consider a Poisson Point Process (PPP) in $\S^1$ with constant intensity equal to $n$, we obtain the desired independence for the number of points in two disjoint intervals. The argument finishes by showing that the PPP is a good approximation of our model (but we will do this fritstly).

Consider an infinite sequence of independent uniform random variables in $\S^1$, $x_0, x_1, \dots$. Let $\n$ be an independent Poisson random variable with parameter $n$. Define for every $k \in \N$, $V_k=\{x_0, \dots, x_{k-1}\}$. Then $V=V_n$ and we denote $\mathsf V:= V_\n$. The point process $\mathsf V$ is a PPP in $\S^1$. Let us consider the Poissonized version of the energy,
\[
\mathsf E(u)= \mathsf E(u_0,\dots, u_{\n-1}) = \displaystyle{\frac{\pi}{2 n^{2} \ep_n^3}} \sum_{i=1}^{\n}  \sum_{\substack{ j\sim i\\x_j \in \mathsf V}}(1 - \cos (u_j-u_i)).\\
\]
%Here $\mathbf N_i$ is the number of neighbors of $i$ in the graph constructed with $\mathsf V$ instead of $V$.
The first goal is to prove that $\mathsf E(u) - E_n(u) \to 0$ in probability as $n\to \infty$. Once we do that, we can restrict ourselves to proving the proposition for the Poissonized version of the energy $\mathsf E(u)$.

To do that, we need to consider versions of the energy for different sets of nodes. So, define for any $V_k$,
\[
E^{V_k}_n(u)= E^{V_k}_n(u_0,\dots, u_{k-1}) = \displaystyle{\frac{\pi}{2 n^{2} \ep_n^3}} \sum_{i=1}^{k}  \sum_{\substack{j\sim i \\x_j \in {V_k}}}(1 - \cos (u_j-u_i)).\\
\]
%Similarly, $N^{V_k}_i$ denotes the number of neighbors of $i$ in the graph constructed with $V_k$. 
With this notation we have $E_n = E^{V_n}_n$ and $\mathsf E = E^\mathsf V_n$. Observe that here the subindex $n$ identifies the parameters used in the scaling and the set of nodes in the superindex dictates the graph used to determine which are the terms in the sum. For every $k \ge -n +2$,
\begin{align*}
 |E^{V_{n+k}}_n(u) - E^{V_n}_n(u)|  & \le \displaystyle{\frac{\pi}{2 n^{2} \ep_n^3}} \sum_{i=1}^{n \vee (n+k)} \sum_{\substack{j\sim i \\x_j \in {V_n}\Delta V_{n+k}}}(1 - \cos (u_j-u_i))\\
 & \le \displaystyle{\frac{\pi}{2 n^{2} \ep_n^3}} \sum_{i=1}^{n \vee (n+k)} \sum_{\substack{j\sim i \\x_j \in {V_n}\Delta V_{n+k}}}\frac{1}{2}(u_j-u_i)^2\\
 & \le \displaystyle{\frac{\pi}{4 n^{2} \ep_n^3}} \sum_{i=1}^{n \vee (n+k)} \sum_{\substack{j\sim i \\x_j \in {V_n}\Delta V_{n+k}}}\|(u')^2\|_\infty \ep_n^2.
\end{align*}

Here $\Delta$ denotes symmetric difference $A\Delta B = (A \cup B)\backslash (A \cap B)$. Since for each $i, j$ we have $\P(i\sim j) = \ep_n/\pi$, we get
\[
\E |E^{V_{n+k}}_n(u) - E^{V_n}_n(u)|  \le \displaystyle{\frac{\pi}{2 n^{2} \ep_n}} \|(u')^2\|_\infty ({n \vee (n+k)}) |k| \frac{\ep_n}{\pi},
\]
and for $k=\mathsf N -n$ we obtain,
\[
\E |\mathsf E(u) - E_n(u)|  \le \frac{\|(u')^2\|_\infty \E \left [(n  + \mathsf N) |\mathsf N - n| \right ]}{{2 n^{2}}} \to 0.
\]
%
% There is a random variable $C$ independent of $n$ and $u$ such that for every integer $k \ge -n +2$,
% \[
% |E^{V_{n+k}}_n(u) - E^{V_n}_n(u)| \le C\frac{|k|}{n}\|(u')^2\|_\infty.
% \]
%
% \ceci{Para esta cota considerar $L_{n,k}$ = max de puntos en un intervalo de longitud $2\ep_n$ donde el max se toma sobre todos los centros de los intervalos posibles. La cota quedar\'ia $\leq (n+|k|) \ep_{n} \log n$}.
% In particular,
% \[
% |\mathsf E(u) - E_n(u)| \le C\frac{|\n -n|}{n}\|(u')^2\|_\infty.
% \]
This implies that $\mathsf E(u) - E_n(u) \to 0$ in probability and hence it is enough to prove
\[
\lim_{n\to \infty} \mathsf E(u) = \frac{1}{12\pi} \int_0^{2\pi}|u'(x)|^2\, dx, \qquad \text{in probability}
\]

to obtain the proposition. We proceed to do that. Using first order Taylor expansion of $u$ around $x$, we get

\begin{align*}
\frac 1 \ep \int_0^{2\pi} \int_0^{2\pi}  \left( \frac{u(y)-u(x)}{\ep} \right)^2 &\mathbf 1\{|y-x|<\ep \} \, dy\,dx =  \\
& = \frac 1 \ep \int_0^{2\pi} \int_{x-\ep}^{x+\ep}  \left( \frac{u'(x) (y-x) + \frac12 u''(c_x) (y-x)^2  }{\ep} \right)^2 dy\,dx \\
& =  \frac{1}{\ep}\int_0^{2\pi} \frac{u'(x)^2}{\ep^2}  \left( \int_{x-\ep}^{x+\ep} (y-x)^2 dy \right) dx + O(\ep) \\
& = {\frac23} \int_0^{2\pi}|u'(x)|^2\, dx + O(\ep).
\end{align*}
Moreover, since $1-\cos(t) = \frac{t^2}{2} + O(t^4)$ as $t\to0$, we have,
\begin{equation}
\lim_{\ep \to 0} \frac{1}{\ep}\int_0^{2\pi} \int_0^{2\pi}  \left( \frac{1-\cos(u(y)-u(x))}{\ep^2} \right) \mathbf 1\{|y-x|<\ep \} \, dy\,dx =  \frac13 \int_0^{2\pi}|u'(x)|^2\, dx.
\end{equation}
% \begin{equation}
% \E[(u(x_j)-u(x_i))^2\mathbf 1\{x_j\sim x_i\}|x_i=x, \mathcal N(i)] =  \frac{1}{2\ep_n}\int_0^{2\pi}  \left( {u(y)-u(x)} \right)^2 \mathbf 1\{|y-x|<\ep_n \} \, dy \mathbf 1\{|x_j-x|<\ep_n\}
% \end{equation}

When a random vector uniformly distributed in a (Borel) set $A$ is conditioned to belong to a subset $B \subset A$, its distribution is still uniform (in $B$). As a consequence, if we choose a set of i.i.d. variables uniformly distributed in $A$ and condition to the event that all of them belong to $B$, they are still i.i.d. uniform (in $B$). This can be seen by applying the previously mentioned result in the product space ($k$ i.i.d. variables uniformly distributed in $A$ is equivalent to being uniformly distributed in $A^k$). We will use this fact in the following lines.

Let ${\mathcal N}(i)$ be the set of neighbors of $i$ in the graph determined by $\mathsf V$. That is,
\[
\mathcal N(i) := \{j\ne i\colon 0\le j \le \n, \,  d(x_i,x_j)<\ep_n\}.
\]

For $i,j \in \N$, $i\ne j$ we define the random variables

\begin{align*}\label{def-xi}
\zeta_j^i&:= [1 - \cos(u(x_j)-u(x_i))]\mathbf 1 \{ j\in \mathcal N(i)\}\\
\nonumber & = [1 - \cos(u(x + \ep_n z^i_j)-u(x))]\mathbf 1 \{|z^i_j|<1 \}.
%\quad\quad \quad j\in\cN(i) \nonumber
\end{align*}

Here $z^i_j:= (x_j-x_i)/\ep_n$. For a given $i \in \N$ we condition on the event %$\{i\le \n\} \cap
$\{x_i=x\}$ and the random variable $\mathbf{\mathcal N}(i)$ (i.e. the labels of the points that are at distance less than $\ep_n$ from $x$). Observe that if $\mathcal A \subset \N$ is a set of labels, conditioning on $\mathcal N(i)=\mathcal A$ is equivalent to conditioning on the event that the nodes $x_j$ with index $j\in \mathcal A$ belong to $(x-\ep_n,x+\ep_n)$. Under this conditioning, the variables $z^i_j$, $j\in\cN(i)$ are i.i.d. uniform in $[-1,1]$. Moreover, the variables $\{ \zeta_j^i \colon j \in \mathcal N(i) \}$ are also i.i.d. and their absolute values are bounded by $\|(u')^2\|_\infty \ep_n^2$. Their conditional expectation is given by,

\begin{align*}
\E\left(\zeta_j^i\, |\, x_i=x, \cN(i)\right) & = \frac{1}{2}\int_{-1}^1 [1 - \cos(u(x + \ep_n z)-u(x))]\, dz\\
& =  \frac{1}{2\ep_n}\int_{0}^{2\pi} [1 - \cos(u(y)-u(x))] \mathbf 1 \{|y-x|<\ep_n \} \, dy.
%& =  \frac{1}{\sigma_d\ep^d}\int_{\T^d} \sin\left(u^{I,\ep}(y,t)-u^{I,\ep}(x_i,t)\right)K\left(\ep^{-1}(y-x_i)\right) dy.
\end{align*}

Then,

\begin{align*}
\E \left[\frac{\pi}{n \ep_{n}^{3}}\sum_{j\in\cN(i)}\zeta_j^i \right]& = \frac{\pi}{n \ep_{n}^{3}} \, \E(\mathcal{N}(i)) \, \E(\zeta_j^i)\\
& = \frac{\pi}{n \ep_{n}^{3}} \frac{n \ep_n}{\pi} \frac{1}{2 \ep_n} \int_0^{2\pi} \frac{1}{2\pi} \int_{0}^{2\pi} [1 - \cos(u(y)-u(x))] \mathbf 1 \{|y-x|<\ep_n \}\, dy \, dx \\
& = \frac{1}{4\pi}\frac 1 {\ep_n}\int_0^{2\pi}\int_{0}^{2\pi} \left (\frac{1 - \cos(u(y)-u(x))}{\ep_n^2}\right) \mathbf 1 \{|y-x|<\ep_n \}\, dy \, dx,
\end{align*}

and

\begin{align*}
\E \left[\frac{\pi}{n \ep_{n}^{3}} \sum_{j\in\cN(i)}\zeta_j^i \right]^2 \le \E \left[\frac{\pi}{n \ep_{n}^{3}} |\cN(i)| \|(u')^2\|_\infty \ep_n^2 \right]^2 \le \|(u')^2\|_\infty^2.
\end{align*}

Define the variables

\begin{equation}
\label{Zin}
Z^n_i:=\frac{\pi}{n \ep_{n}^{3}}\sum_{j\in\cN(i)}\zeta_j^i, \qquad 0 \le i \le \n-1,
\end{equation}
and observe that $\mathsf E(u) = \frac 1 n \sum_{i=1}^{\mathsf N} Z^n_i$. Lemma \ref{LLN} below proves that $\mathsf E(u) \to \mu:= \frac{1}{12\pi} \int_0^{2\pi}|u'(x)|^2\, dx$ and hence the same holds for $E_n(u)$. This concludes the proof.

\end{proof}

For the sake of well-definiteness we construct $Z^n_i$ for $i\ge\n$ using independent copies of the process.

\begin{lemma}
\label{lemma.cov}For $Z_i^n$ as defined in \eqref{Zin} we have, for $i\ne j$,
\[
|{\rm Cov}(Z_i^n, Z_j^n)| \le \frac{4\|(u')^2\|^2_\infty \ep_n}{\pi}.
\]
\end{lemma}

\begin{proof} Proceeding as before we get,
\[
\E[Z_i^n] \le   \|(u')^2\|_\infty \qquad \text{and} \qquad \E[(Z_i^n)^2|x_i, x_j, \cN(i), \cN(j)] \le   \|(u')^2\|^2_\infty.
\]

We compute for $i\ne j$,
\begin{align*}
\E[Z_i^n Z_j^n| &  x_i, x_j, \cN(i), \cN(j) ] = \\
& = \E\left[Z_i^n Z_j^n \mathbf 1\{|x_i-x_j|> 2\ep_n\}|x_i, x_j, \cN(i), \cN(j) \right] \\
& +  \E\left[Z_i^n Z_j^n \mathbf 1\{|x_i-x_j| \le 2\ep_n\}|x_i, x_j, \cN(i), \cN(j) \right].\\
\end{align*}
But,
\begin{align*}
\E \! \left[Z_i^n Z_j^n \mathbf 1\{|x_i-x_j|> 2\ep_n\}|x_i, x_j, \cN(i), \cN(j) \right] \! = \!\E \!\left[Z_i^n|x_i, \cN(i)\right] \! \E \!\left[Z_j^n|x_j, \cN(j)\right] \mathbf 1\{|x_i-x_j|> 2\ep_n\}.
\end{align*}

This can be seen by consdering the two cases $|x_i-x_j|> 2\ep_n$ and $|x_i-x_j|\le 2\ep_n$. In the first case $Z_i^n$ and $Z_j^n$ are (conditionally) independent and in the second case we obtain zero on both sides of the equality.

By H\"older's inequality,
\begin{align*}
\E\left[Z_i^n Z_j^n \mathbf 1\{|x_i-x_j| \le 2\ep_n\}|x_i, x_j, \cN(i), \cN(j)\right ] \le \|(u')^2\|^2_\infty \mathbf 1\{|x_i-x_j| \le 2\ep_n \}.
\end{align*}
So that,
\[
\E\left[Z_i^n Z_j^n \mathbf 1\{|x_i-x_j| \le 2\ep_n\}\right ] \le  \|(u')^2\|^2_\infty \P(|x_i-x_j| \le 2\ep_n ) \le \|(u')^2\|^2_\infty \frac{2 \ep_n}{\pi}.\\
\]

Hence,

\begin{align*}
\big|\E(Z_i^n Z_j^n)& - \E(Z_i^n)\E(Z_j^n) \big|= \\
& = \left| \E(Z_i^n)\E(Z_j^n)\P(|x_i-x_j|> 2\ep_n) + \E\left[Z_i^n Z_j^n \mathbf 1\{|x_i-x_j| \le 2\ep_n\}\right] - \E(Z_i^n)\E(Z_j^n) \right|\\
& = \left| - \E(Z_i^n)\E(Z_j^n)\P(|x_i-x_j| \le 2\ep_n) + \E\left[Z_i^n Z_j^n \mathbf 1\{|x_i-x_j| \le 2\ep_n\}\right] \right| \\
& \le \frac{ 4 \|(u')^2\|^2_\infty \ep_n}{\pi}.
\end{align*}
\end{proof}

\begin{lemma}
\label{LLN}
For $Z^n_1, Z_2^n, \dots, Z_\n^n$ defined as above we have,
%with $\E(Z_1^n) \to \mu$, $\E(Z_1^n)^2\le C$ and ${\rm Cov}(Z_1^n, Z_2^n) \to 0$ as $n\to \infty$. If $\n$ is an independent Poisson random variable with mean $n$ Then,
\[
\frac1n \sum_{i=1}^\n Z_i^n \to \mu= \frac{1}{12\pi} \int_0^{2\pi}|u'(x)|^2\, dx, \qquad \text{ in probability.}
\]
\end{lemma}
\begin{proof}
Call $\bar Z_n = \frac1n \sum_{i=1}^\n Z_i^n$. We compute,

\begin{align*}
\E\left[\frac1n \sum_{i=1}^\n Z_i^n\right]^2 & = \frac{\E(\n)}{n^2}\E(Z_1^n)^2 +  \frac{1}{n^2}\E\left[\sum_{i\ne j}^\n Z_i^n Z_j^n\right]\\
& = \frac{1}{n}\E(Z_1^n)^2 +  \frac{\E(\n(\n-1))}{n^2} {\rm Cov}(Z_1^n, Z_2^n) + \frac{\E(\n(\n-1))}{n^2}\E^2(Z_1^n).
\end{align*}
We have that $\E(Z_1^n)^2\le \|(u')^2\|_\infty^2$ and ${\rm Cov}(Z_1^n, Z_2^n) \to 0$. Then $\E(\bar Z_n^2) \to \mu^2$. Since $\E(\bar Z_n) \to \mu$, the variance ${\rm Var}(\bar Z_n)\to 0$. By means of Tchebychev's inequality, $\bar Z_n \to \mu$ in probability.
\end{proof}

We are ready to prove the main theorem. Although it will require some technicalities, the idea of the proof is simple. For a given $q \in \Z$, we are looking for a local minimum of $E_n$ with index $q$. Since $\bar K_q = K_q \cup \partial K_q$ is compact and $E_n$ is continuous, we have the existence of a minimum of $E_n$ in $\bar K_q$. In order to guarantee that this minimum is in fact a local minimum of $E_n$ we need to show that it is not on the boundary $\partial K_q$. We will do that by proving that
\begin{enumerate}
 \item with high probability there is a point $u_q \in K_q$ with bounded (in $n$) energy.
 \item the minimum of the energy along the boundary $\partial K_q$ goes to infinity as $n\to\infty$.
\end{enumerate}
Statements (1)+(2) imply that for large $n$ the minimum of $E_n$ can not lie on the boundary and hence it is a local minimum of the energy. We now proceed with the details.

\begin{proof}[Proof of Theorem \ref{main.thm}]

First, by Bernstein's inequality and union bound, we have that

\begin{align}\label{bernstein}
\P\left(\sup_{i=1}^n|N_i-\frac{\ep_n n}{\pi}|>\lambda\right)\le 2n e^{-\frac{\frac12\lambda^2}{{\ep_n n}/{\pi}+\lambda/3}},\quad \lambda>0.
\end{align}

For $\lambda = \frac{\ep_n n}{\pi}$ we obtain
\begin{align}\label{bernstein2}
\P\left(\sup_{i=1}^n N_i \ge \frac{2\ep_n n}{\pi}\right)\le 2n e^{-c\ep_nn}.
\end{align}

Similarly, if we call $N_{ij} = |\{k \colon |x_i - x_k|< \ep_n, |x_{j} - x_k|< \ep_n\}|$ the number of common neighbors of $i$ and $j$, we have $\E(N_{ij} | {i} \sim {j}) \ge \frac{n \ep_{n}}{2\pi}$ and

\begin{align*}
\P\left(N_{ij} \le \frac{\ep_n n}{4\pi}\Big | i \sim j \right)\le e^{-c\ep_nn}.
\end{align*}

Hence,
\begin{align}\label{bernstein3}
 \P\left(\inf_{i \sim j} N_{ij} < \frac{n \ep_{n}}{4 \pi}\right) &\le \sum_{i,j} \P\left(N_{ij} < \frac{n \ep_{n}}{4 \pi} \Big| {i} \sim {j}\right) \P({i} \sim {j}) \le n^2 e^{-c \ep_{n} n} (\ep_{n}/\pi).
\end{align}

Let $\mathbf z \in {\mathbb{T}^{n}}$, $\mathbf z = (z_0,\dots, z_{n-1}) = (e^{{\rm i} \theta_0}, \dots, e^{{\rm i} \theta_{n-1}})$ such that $I(\mathbf z)$ is not defined. Then there is $k$ with $z_k=-z_{k-1}$ and hence  we have $\cos(\theta_{k-1}-\theta_k)=-1$. For any $\theta \in [0,2\pi)$ we have $\cos(\theta_{k-1}-\theta) \wedge \cos(\theta_{k}-\theta) \le 0$.  If  $G$ is connected and $r$ is a neighbor of both $k$ and $k-1$ we have
\[
\sum_{j\sim r}(1 - \cos (\theta_j-\theta_r)) \ge  1.
\]
Hence
\begin{equation}
\label{bound.below.energy}
E_n(\theta_0, \dots, \theta_{n-1}) =  \displaystyle{\frac{\pi}{2 n^2\ep_n^3}} \sum_{i=1}^{n}  \sum_{j\sim i}(1 - \cos (\theta_j-\theta_i)) \ge  \displaystyle{\frac{ \pi }{2n^2 \ep_n^3}} \, N_{k,k-1}.
\end{equation}

Due to \eqref{bernstein3} we have for every $q\in \Z$ ,

\begin{align*}
\P \left( \displaystyle \inf_{\bm \theta \in \partial K_q} E_n(\bm{\theta}) \le \displaystyle{\frac{1}{8n \ep_n^2}} \right ) &\le  \P\left(\inf_{i \sim j} N_{ij} \le \frac{\ep_n n}{4\pi}\right) + \P\left(\bigcup_{i=1}^{n} \{{i} \nsim {i-1}\}\right) \\
&\le %2ne^{-c\ep_nn} +
n^2e^{-c\ep_nn} + n \left(\frac{\pi - \ep_{n}}{\pi}\right)^{n-1}.
\end{align*}

Condition \eqref{condition} guarantees $n \ep_n^2 \to 0$ and that $\sum_n n^{2} e^{-c\ep_nn}< \infty$. Thus,
\begin{equation}
\label{boundary.to.infinity}
\lim_{n\to\infty} \displaystyle \inf_{(e^{{\rm i} \theta_0}, \dots, e^{{\rm i} \theta_{n-1}}) \in \partial K_q} E_n(\theta_0,\dots, \theta_{n-1}) = +\infty, \qquad \text{a.s.}
\end{equation}
%
% \begin{lemma}
% Assume $u^n$ is such that $I(u^n)$ is not defined, then $E_n(u^n)\to \infty$ a.s.
% \end{lemma}
For $q\in \Z$ we consider the function $u_q(x)=qx$. Observe that $(u_q(x_0), \dots, u_q(x_{n-1})) \in K_q$. By means of Proposition \ref{conv.grad} we compute
\[
\lim_{n\to\infty}E_n(u_q) = \frac{1}{12\pi}\int_0^{2\pi}|u'_q(x)|^2\, dx =\frac{q^2}{6}, \qquad \text{ in probability}.
\]
Define the event
\[
A_{n,q} := \left \{\inf_{\mathbf z \in \partial K_q} E_n(\mathbf z) > \frac{q^2}{4} \text{ and } E_n(u_q) < \frac{q^2}{5} \right\}.
\]
Proposition \ref{conv.grad} guarantees that
\[
\P\left( E_n(u_q) \ge \frac{q^2}{5} \right) \to 0,
\]
and \eqref{boundary.to.infinity} gives us
\[
\P\left(\inf_{\mathbf z \in \partial K_q} E_n(\mathbf z) \le \frac{q^2}{4}\right) \to 0.
\]
Combining these two facts we get $\P(A_{n,q}) \to 1$ as $n\to \infty$. Finally, observe that since $\overline{K_q}$ is compact and $E_n$ is continuous, it attains a minimum at $\overline{K_q}$. If $A_{n,q}$ occurs this minimum can not be attained in $\partial K_q$ and hence there is a point $u_q^* \in K_q$ with
\[
E_n(u_q^*) \le E_n(u), \qquad \text{ for every } u \in K_q.
\]
Since $K_q$ is open, $u_q^*$ is a local minimum of $E_n$ and hence a stable equilibrium for \eqref{eq:km}. We have proved that for every $q\in \Z$
\[
\P\left(\eqref{eq:km} \text{ has a stable equilibrium with index $q$} \right) \ge \P(A_{n,q}) \to 1.
\]
To ensure that $u_q^*$ is a strict local minima and hence asymptotically stable,  we verify a well-known condition that implies that the Hessian $D^2E_n(u_q^*)$ is positive definite, namely
\begin{equation}
\label{regularity.uq}
|u_q^*(x_i) -  u_q^*(x_j)| < \frac{\pi}{2}, \qquad \text{ for every } i\sim j,
\end{equation}
(see \cite{ling2019landscape}). Let $B_{n,q}:= \{ \omega \colon \omega \in A_{n,q} \text{ and \eqref{regularity.uq} does not hold} \}$.  If there is $k\sim \ell$ with $|u_q^*(x_k) -  u_q^*(x_\ell)| \ge \frac{\pi}{2}$, proceeding as in \eqref{bound.below.energy} we bound from below
\begin{equation*}
E_n(u_q^* ) =  \displaystyle{\frac{\pi}{2 n^2\ep_n^3}} \sum_{i=1}^{n}  \sum_{j\sim i}(1 - \cos (u_q^*(x_i)-u_q^*(x_j))) \ge  \displaystyle{\frac{\pi}{n^2 \ep_n^3}} \, N_{k,\ell}.
\end{equation*}
Since in $A_{n,q}$ we have $E_n(u_q^*)\le E_n(u_q) \le q^2/5$, using Bernstein's inequality again we obtain for $n$ large enough,
\begin{align*}
 \P(B_{n,q}) & \le \P \left(A_{n,q} \cap \left\{ E_n(u_q^* ) \ge  \inf_{i \sim j} \frac{\pi}{n^2 \ep_n^3} \, N_{i,j} \right\} \right)\\
 & \le %\P\left(\sup_{i} N_i \ge \frac{2\ep_n n}{\pi}\right) +
 \P\left(\inf_{i \sim j}     N_{ij} \le \frac{\ep_n n}{4\pi}\right)\\
 & \le n^2 e^{-c\ep_nn}.
\end{align*}
Since $\P(A_{n,q}) \to 1$, we get that
\[
 \P(A_{n,q} \text{ and \eqref{regularity.uq} holds}) \to 1.
\]
In particular,
\[
\lim_{n\to\infty}\P\left(\eqref{eq:km} \text{ has an asymptotically stable equilibrium with index $q$} \right)  = 1.
\]
\end{proof}

\section{Discussion}
\label{discussion}

In this section we discuss other models for which our results should still hold, possible extensions and other considerations.

\subsection{Other graph models.} In view of the proof of Theorem \ref{main.thm}, that is based on the convergence of $E_n(u)$ for smooth functions $u$ and the fact the $E_n$ goes to infinity at the boundary of each $K_q$, we also expect the same result to hold for the following families of graphs. All are based on nodes $V=\{x_0, \dots, x_{n-1}\}$ i.i.d uniformly distributed on the unit circle. Different models correspond to different sets of edges.

\bigskip

{\bf a. $k-$nn graphs} Two vertices $x_i$ and $x_j$ are connected by an edge if the distance between $x_i$ and $x_j$ is among the $k_n-$th smallest distances from $x_i$ to other nodes from $x_i$ or vice versa. Condition \eqref{condition} becomes
\[
\frac{k_n^2}n \to 0, \qquad  \frac{k_n}{\log n} \to \infty.
\]

\bigskip

{\bf b. Boolean model.} For each node $x_i$ we consider a random radius $r_i$. We assume the radii are i.i.d. We declare two nodes $x_i$, $x_j$ neighbors if
\[
(x_i-r_i,x_i+r_i) \cap (x_j-r_j,x_j+r_j) \ne \emptyset.
\]
The role of $r_i$ is similar to the one of $\ep_n/2$ in the original model but now they are random. Condition \eqref{condition} becomes
\[
 n\E(r_i^{2})\to 0, \qquad  \frac{n\E(r_i)}{\log n} \to \infty,  \qquad \text{ as }n\to\infty.
\]

\bigskip

{\bf c. Random $N-$nn.} This is similar to the $k-$nn graph but instead of considering a deterministic $k$ we choose a random number $N_i$ for each $x_i$. The variables $(N_i)_{0\le i\le n-1}$ are i.i.d.

\bigskip

{\bf d. Weighted graphs.} In any of the previous models or even in WSG networks (with small $k$) we can consider (random or deterministic) weights as far as they don't degenerate as $n\to\infty$. To get a tractable model it is better to consider a kernel $k:\R\to\R_{\ge 0}$ to be a symmetric, smooth function with compact support in $(-1,1)$ and $\int k(z)dz=1.$
% We denote
% $$
% \kappa := \|K\|_\infty,  \qquad \kappa_i := \frac{1}{\sigma_d}\int|z|^iK(z)\, dz, \quad i=1,2.
% $$
Then we consider the weighted graph $G=(V,E)$, where the weights are given by $w_{ij}=k\left(\ep_n^{-1}(x_j-x_i)\right)$. For these graphs, condition \eqref{condition} remains unchanged.

% {\bf 5. Geometric Erd\H{o}s-R\'{e}ny graphs} Each edge of the complete graph is retained independently with probability
% \[
%  p_{ij}\propto {k(|x_i-x_j|/\ep_n)}.
% \]
% The kernel $k$ can be Gaussian $k(x)=e^{-\frac{x^2}{2}}$ or the indicator function....

\bigskip

{\bf e. Random geometric graphs in an $\ep_n-$neighborhood of a simple closed curve.}
Consider a simple closed curve $\gamma$ and its $\ep_n$ neighborhood
\[
 \gamma^{\ep_n} :=\{ x \in \R^d \colon d(x,\gamma)<\ep_n\}.
\]
Here $d(x,\gamma)=\inf_{y\in\gamma}|x-y|$. For $\ep_n$ small enough $\gamma^{\ep_n}$ is homeomorphic to an $\ep_n-$neighborhood of the unit circle $C^{\ep_n}$ and we can work on that setting without loss of generality. So, consider in $\R^d$ the set $C^{\ep_n}$ with
\[
C = \{(x,y,0,  \dots,0)\in \R^d \colon x^2 + y^2=1\}.
\]
We consider as in the whole manuscript a sample $V=\{x_0, \dots, x_{n-1}\}$ of $n$ i.i.d. uniform points in $C^{\ep_n}$ and we declare $x_i\sim x_j$ if and only if their projections in the unit circle are at distance less than $\ep_n$. Observe that this implies that the distance between them is less than $3\ep_n$. By working with the projections, we obtain a random geometric graph in the circle and hence we can apply Theorem \ref{main.thm}.

\subsection{Bounds for the existence of $u_q$}

In the course of the proof of Theorem \ref{main.thm} we saw that with high probability the infimum of the energy on the boundary of any $K_q$ is bounded below by $(8n\ep_n^2)^{-1}$. This bound is sharp. Then we expect the event $A_{n,q}$ to have small probability for $(8n\ep_n^2)^{-1} < q^2/4$ and large probability when $(8n\ep_n^2)^{-1} > q^2/4$, which is equivalent to
\[
|q| < \frac{1}{2\sqrt n\ep_n} \to \infty.
\]
Hence, the larger the $|q|$, the larger the $n$ we need to get the existence of a $q-$twisted state with high probability.

In fact, following the same arguments it can be proved that if $q_n < \frac{1}{2\sqrt n\ep_n}$ for $n$ large enough, then
\[
\lim_{n\to\infty}\P\left(\eqref{eq:km} \text{ has an asymptotically stable equilibrium with index $q_n$} \right)  = 1.
\]

\subsection{The role of the scaling factor.} Equation \eqref{eq:km} is scaled according to the factor ${1}/{n^2\ep_n^3}$. The goal of this factor is to obtain Proposition \ref{conv.grad}, but once we obtain the existence of $q-$twisted states for a specific value of $n$, the scaling factor plays no role and the same conclusion can be obtained for any other constant used to normalize the energy $E_n$.

\subsection{Higher dimensions.} We discuss now the extension of our results to higher dimensions, as in the spirit of \cite{CGHV}. Instead of the circle $\S^1$, we assume that the set of nodes $V$ is given by $n$ i.i.d. points in the $d-$dimensional torus $\T^d$. The set of edges is defined in the same way: $x_i \sim x_j$ if and only if $d(x_i,x_j)< \ep_n$.

The definition of winding number is specially suited for dimension one as it relies strongly on the fact that $\S^1$ can be parametrized with a curve. However, a notion of winding number can be given for each canonical vector. In this context the winding number of a function $u\colon V \to \S^1$ would be a $d-$dimensional vector rather than a number \cite{CGHV}. Then, a similar partition of the space can be carried out as in Section \ref{geometry.space}.

When working in higher dimensions, the scaling of the energy should be
\[
E_n(u)= E_n(u_0,\dots, u_{n-1}) =  \frac{\pi}{2 n^{2} \ep_{n}^{d+2}}\sum_{i=1}^{n} \sum_{j\sim i}(1 - \cos (u_j-u_i)).\\
\]
A result equivalent to Proposition \ref{conv.grad} can be obtained similarly under the condition
\[
\ep_n\to 0, \qquad  \frac{n\ep_n^d}{\log n} \to \infty,  \qquad \text{ as }n\to\infty.
\]
The problem appears when we want to bound from below the infimum of the energy at the boundary of $K_q$. Following \eqref{bound.below.energy} we get that if $(\theta_0,\dots, \theta_{n-1}) \in \partial K_q$,
\begin{equation}
\label{bound.below.energy.d}
E_n(\theta_0, \dots, \theta_{n-1}) =  \displaystyle{\frac{\pi}{2 n^2\ep_n^{d+2}}} \sum_{i=1}^{n}  \sum_{j\sim i}(1 - \cos (\theta_j-\theta_i)) \ge  \displaystyle{\frac{ \pi }{2n^2 \ep_n^{d+2}}} \, N_{k,k-1} \approx \displaystyle{\frac{ \pi }{2n^2 \ep_n^{d+2}}}n\ep_n^d.
\end{equation}
So, the condition to guarantee that the infimum of the energy at the boundary of $K_q$ goes to infinity is still
\[
n\ep_n^2 \to 0,
\]
which is not compatible with $n\ep_n^d/\log n \to \infty$ (unless $d=1$).

It is somehow curious that although our results hold only in dimension one, this does not seem to be related to the geometry or specific properties of one-dimensional spaces but just to the scaling of the exponents.

To get a result similar to Theorem \ref{main.thm} for dimensions $d\ge 2$ with this method, it would be necessary either to obtain a better lower bound in \eqref{bound.below.energy.d} or to find another argument to discard that the minimum obtained by compactness is at the boundary.

\bigskip

{\textbf{Acknowledgments.}}
We thank Steven Strogatz for illuminating discussions. Pablo Groisman and Cecilia De Vita are partially supported by CONICET Grant PIP 2021 11220200102825CO, UBACyT Grant 20020190100293BA and PICT 2021-00113 from Agencia I+D.

Juli\'an F. Bonder is partially supported by CONICET under grant
PIP 11220150100032CO and PIP 11220210100238CO and by ANPCyT under grants
PICT 2019-3837 and PICT 2019-3530.

\bibliographystyle{plain}
\bibliography{kuramotoSL}
\end{document}